\documentclass[a4paper, reqno]{amsart}
\usepackage[margin=1.2in]{geometry}
\allowdisplaybreaks

\title[]{Minimal $\W$-algebras of $\mathfrak{so}_N$ at level minus one}



\author[T. Creutzig]{Thomas Creutzig}
\address[T.C.]{Department Mathematik, FAU Erlangen–Nürnberg, Cauerstraße 11, 91058, Erlangen, Germany}
\email{creutzigt@math.fau.de}

\author[J. Fasquel]{Justine Fasquel}
\address[J.F.]{Université Bourgogne Europe, CNRS, IMB UMR 5584, 21000 Dijon, France}
\email{justine.fasquel@u-bourgogne.fr}

\author[V. Kovalchuk]{Volodymyr Kovalchuk}
\address[V.K.]{Department Mathematik, FAU Erlangen–Nürnberg, Cauerstraße 11, 91058, Erlangen, Germany}
\email{kovalchukv@math.fau.de}

\author[A. R. Linshaw]{Andrew R. Linshaw}
\address[A.L.]{Department of Mathematics, University of Denver, C. M. Knudson Hall, 2390 S. York St. Denver, CO 80210}
\email{andrew.linshaw@du.edu}

\author[S. Nakatsuka]{Shigenori Nakatsuka}
\address[S.N.]{Laboratoire de Mathématiques d’Orsay\\Université Paris-Saclay, 91405 Orsay, France}
\email{shigenori.nakatsuka@universite-paris-saclay.fr}

\usepackage{stmaryrd}
\usepackage{array}
\usepackage{latexsym}
\usepackage{amsmath}
\usepackage{amsfonts}
\usepackage{amssymb,nccmath}
\usepackage{mathtools}
\usepackage{amsxtra}
\usepackage{mathrsfs}
\usepackage{eucal}
\usepackage[all]{xy}
\usepackage{color}
\usepackage{braket}
\usepackage{graphics, setspace}
\usepackage{etoolbox, textcomp}
\usepackage{comment}
\usepackage{changepage}
\usepackage{xcolor}
\definecolor{rouge}{rgb}{0.85,0.1,.4}
\definecolor{bleu}{rgb}{0.1,0.2,0.9}
\definecolor{violet}{rgb}{0.7,0,0.8}
\usepackage[colorlinks=true,linkcolor=bleu,urlcolor=violet,citecolor=rouge]{hyperref}
\usepackage{cleveref}
\usepackage{lscape}
\usepackage{caption}
\usepackage{subcaption}
\usepackage{enumitem}
\usepackage{graphicx}
\usepackage{arydshln}
\usepackage{tikz} \usetikzlibrary{calc} \tikzset{>=latex} \usetikzlibrary{backgrounds} \usetikzlibrary{shapes.geometric}
\usepackage{tikz-cd}
\usepackage[colorinlistoftodos]{todonotes}
\usepackage{bbm}
\usepackage{blindtext}

\usepackage[aligntableaux=center]{ytableau}
\ytableausetup{mathmode,boxsize=0.4em}

\usepackage{tikz}			
\usetikzlibrary{matrix}
\usetikzlibrary{automata, arrows.meta, positioning,cd}
\usepackage{dynkin-diagrams}
\usetikzlibrary{cd}			

\newtheorem{definition}{Definition}[section]
\newtheorem{proposition}[definition]{Proposition}
\newtheorem{theorem}[definition]{Theorem}

\newtheorem{corollary}[definition]{Corollary}
\newtheorem{lemma}[definition]{Lemma}
\newtheorem{conjecture}[definition]{Conjecture}
\theoremstyle{remark}
\newtheorem{remark}[definition]{Remark}

\numberwithin{equation}{section}

\newcommand{\Z}{\mathbb{Z}}     
\newcommand{\C}{\mathbb{C}}     

\newcommand{\sC}{\mathcal{C}}
\newcommand{\sSC}{\mathcal{SC}}
\newcommand{\bX}{\mathbb{X}}
\newcommand{\bY}{\mathbb{Y}}

\newcommand{\W}{\mathcal{W}}    

\newcommand{\scV}{\mathscr{V}}    

\newcommand{\HH}{\mathrm{H}}    
\newcommand{\OO}{\mathbb{O}}    
\newcommand{\oo}[1]{\OO_{[#1]}}
\newcommand{\minO}{\mathbb{O}_{\mathrm{min}}}

\newcommand{\lbr}[2]{ {[} {#1} {}_\lambda {#2} {]}}

\newcommand{\lm}{\lambda}

\newcommand{\pd}{\partial}

\newcommand{\End}{\operatorname{End}}
\newcommand{\Hom}{\operatorname{Hom}}

\renewcommand{\End}{\mathrm{End}}
\renewcommand{\min}{\mathrm{min}}

\newcommand{\g}{\mathfrak{g}}   
\newcommand{\h}{\mathfrak{h}}   

\newcommand{\sll}{\mathfrak{sl}}    
\newcommand{\spp}{\mathfrak{sp}}    
\newcommand{\so}{\mathfrak{so}}    
\newcommand{\gl}{\mathfrak{gl}}     
\newcommand{\osp}{\mathfrak{osp}}

\newcommand{\od}[1]{#1_{\overline{1}}}
\newcommand{\ev}[1]{#1_{\overline{0}}}

\newcommand{\Vir}{\mathrm{Vir}}     

\renewcommand{\epsilon}{\varepsilon}

\newcommand{\vac}{|0\rangle}
\DeclareMathOperator{\sdim}{sdim}

\DeclareMathOperator{\Frob}{Frob}

\newcommand{\ie}{\textit{ie.}}
\newcommand{\eg}{\textit{eg.}}

\newcommand{\ttwolines}[4]{\begin{tabular}{cl} {#1} & {#2} \\  {#3} & {#4}\end{tabular}}
\newcommand{\threelines}[6]{\begin{tabular}{cl} {#1} & {#2} \\  {#3} & {#4}\\  {#5} & {#6} \end{tabular}}
\newcommand{\ccases}[4]{\begin{cases} #1 &#2 \\ #3 & #4 \end{cases}}

\newcommand{\irr}[2]{\mathrm{Irr}_{#1}\left(#2 \right)}
\newcommand{\obj}[1]{\left\{ #1 \right\}}
\renewcommand{\mod}{\operatorname{-Mod}}

\newcommand{\ub}[1]{\underline{#1}}
\newcommand{\one}{\mathbf{1}}
\newcommand{\id}{\mathrm{id}}

\newcommand\doi[2]{\href{http://dx.doi.org/#1}{#2}}
\newcommand{\arxiv}[2]{arXiv: \href{https://arxiv.org/abs/#1}{#1} [#2]}

\usepackage[normalem]{ulem} 

\begin{document}
\begin{abstract}
For $N \in\mathbb Z_{\geq 7}$ we show that the simple minimal $\mathcal{W}$-algebra of $\mathfrak{so}_N$ at level minus one is isomorphic to the even subalgebra of the tensor product of the simple affine vertex superalgebra of $\mathfrak{osp}_{1|2}$ at level $\frac{N-6}{2}$ with $N-4$ free fermions. In particular, when $N$ is even, this minimal $\mathcal{W}$-algebra is strongly rational as conjectured by Arakawa and Moreau.
\end{abstract}

\maketitle

\section{Introduction}

For a simple Lie algebra $\g$ and a complex number $k$, let $V^k(\g)$ be the universal affine vertex algebra at level $k$ and $L_k(\g)$ its unique simple graded quotient. 
Among all the $V^k(\g)$-modules, which are canonically smooth representations over the affine Lie algebra $\widehat{\g}$ at level $k$, integrable representations at non-negative integer levels $k\in \Z_{\geq0}$ form a finite semisimple category and their characters are modular invariant.
These properties follow from the so-called \emph{strong rationality} of the vertex algebra $L_k(\g)$ \cite{FZ92}.

A vertex algebra $V$ is said to be strongly rational if it is lisse and rational. The latter is equivalent to requesting that the representations of $V$ are semisimple, while the former is a geometric condition stating that the associated variety of $V$ is a point.
Historically, strongly rational vertex algebras have attracted a lot of interest. 
This is mainly due to the fact that their representation categories form modular tensor categories \cite{Hu} (see also \cite{BK} and references therein) and that they provide the symmetry algebras of two-dimensional rational conformal field theories in physics \cite{dFMS}.

While $L_k(\g)$ is lisse (and strongly rational) if and only if $k\in \Z_{\geq0}$, more examples of strongly rational vertex algebras are provided by the $\W$-algebras which are obtained from $V^k(\g)$ via quantum Hamiltonian reductions \cite{KW-QR}.
More precisely, for a nilpotent orbit $\OO$ in $\g$ and $k\in\C$, we denote by $\W^k(\g, \OO)$ the $\W$-algebras associated to $\OO$ at level $k$ and by $\W_k(\g, \OO)$ its unique simple graded quotient.

For admissible levels $k$, generalizing the integrable levels, the associated variety of $L_k(\g)$ is the closure of a nilpotent orbit $\OO$ in $\g$, which immediately implies that $\W_k(\g,\OO)$ is lisse \cite{Ara15}. It is also known to be strongly rational \cite{AvE, McR}.
Aside from the proofs in \textit{loc.~cit.}, the strong rationality can be intuitively explained by the theory of building blocks \cite{CFLN24, CKL24} asserting that $\W_k(\g,\OO)$ is a conformal extension of certain fundamental vertex subalgebras, which are identified at these levels with the strongly rational $\W$-algebras associated with principal nilpotent orbits.

Although strongly rational $\W$-algebras are expected to be rare beyond these levels, a systematic framework for studying rational but non-admissible levels is still to be developed more broadly. So far, only a few scattered examples have been proved based on coincidences; see for example \cite{ACK2,KL19,Kaw18}.
From this perspective, the approach using conformal extensions by building blocks is promising, since the conformal extension of strongly rational vertex algebras is again strongly rational \cite{CMSY}.
Such conformal extensions have been studied extensively recently for instance in the context of collapsing levels, see \eg\ \cite{Ad1, Ad3, Ad2}. 

In \cite{AM15}, it was proved that the $\W$-algebra $\W_{k}(\so_{2n},\OO_{\min})$ for the minimal nilpotent orbit $\OO_{\min}$ for $n\geq2$ is lisse if and only if $k\in\Z_{\geq-2}$ and it was conjectured by Arakawa and Moreau it is strongly rational for $k=-1,-2$.
This conjecture has been proved for $n=4$ in \cite{Ad2,Kaw18} as a particular case of families of special isomorphisms related to the Deligne exceptional series; while the general case for $k=-2$ is obtained through the isomorphism \cite{A+7,Ad4}
\begin{align}\label{intro: collaps}
    \W_{-2}(\mathfrak{so}_N, \OO_{\min})\simeq L_{-2 + \frac{N-4}{2}}(\mathfrak{sl}_2)
\end{align}
when $N\geq8$ is even.
In this paper, we prove the conjecture by Arakawa and Moreau \cite{AM15} for the case $k=-1$ for arbitrary $n\geq4$ using the following orbifold realization.
\begin{theorem}\label{coset construction}
For $N\in\Z_{\geq7}$, there exists an isomorphism of vertex algebras
    \begin{align*}
        \W_{-1}(\so_{N},\minO)\simeq (L_{-1+\frac{N-4}{2}}(\osp_{1|2})\otimes \mathcal{F}^{N-4})^{\Z_2}.
    \end{align*}
\end{theorem}
In particular, when $N$ is even, \ie\ $D$-type case, $L_{-1+\frac{N-4}{2}}(\osp_{1|2})$ is strongly rational (see for instance \cite{AiLin22,CL}), and so is the minimal $\W$-algebra.
\begin{corollary}\label{intro cor: strong rationality}
For, $n\geq4$, $\W_{-1}(\so_{2n},\minO)$ is strongly rational.
\end{corollary}

Regardless of the parity of $N$, the orbifold realization of the minimal $\W$-algebra $\W_{-1}(\so_{N},\minO)$ implies the classification of simple modules (Proposition \ref{Classification of simples via orbifolds}).
By identifying these simple modules with those obtained from the quantum Hamiltonian reduction of simple highest weight modules $L_{-1}(\Lambda)$ of the affine Lie algebra $\widehat{\so}_N$, we obtain the following classification. (See \S \ref{sec: ordinary modules} for details on notation.)
\begin{theorem}[Proposition \ref{simples via reduction}]\label{thm:classification}
For $N\in\Z_{\geq7}$, the complete set of inequivalent simple modules for $\W_{-1}(\so_{N},\minO)$ is given by

\begin{itemize}
    \item (Type $D$) $N=2n$
    \renewcommand{\arraystretch}{2.}
    \begin{align*}
    \begin{array}{ll}
          \HH_{\OO}^0(L_{-1}(\lm \omega_1)),&\ccases{\HH_{\OO}^0(L_{-1}(\omega_3))}{(\lm=0)}{\HH_{\OO}^0(L_{-1}( (\lm-1)\omega_1+\omega_2))}{(\lm\neq0)},\\
         \HH_{\OO}^0(L_{-1}(\lm \omega_1+\omega_{n-1})),&  \HH_{\OO}^0(L_{-1}(\lm \omega_1+\omega_n)). 
    \end{array}
    \end{align*}
    where $\lm$ is an integer satisfying $0\leq \lm\leq n-3$.
\renewcommand{\arraystretch}{1.}
    \item (Type $B$) $N=2n+1$
    \renewcommand{\arraystretch}{2.}
    \begin{align*}
    \begin{array}{ll}
         \HH_{\OO}^0(L_{-1}(\lm \omega_1)),& \ccases{\HH_{\OO}^0(L_{-1}(\omega_3))}{(\lm=0)}{\HH_{\OO}^0(L_{-1}( (\lm-1)\omega_1+\omega_2))}{(\lm\neq0)},\\
         & \HH_{\OO}^0(L_{-1}(\mu \omega_1+\omega_n))
    \end{array}
    \end{align*}
    where $\lm,\mu$ are integers satisfying $0\leq \lm\leq 2n-3$, $0\leq \mu\leq 2n-4$.
\renewcommand{\arraystretch}{1.}
\end{itemize}
\end{theorem}

Combining Theorem \ref{thm:classification} together with some categorical considerations presented in Appendix \ref{appendix}, we obtain the following results.
\begin{corollary}\label{cor:res}
    For $N\in\Z_{\geq7}$,
    \begin{itemize}
        \item The category of ordinary $L_{-1}(\so_n)$-modules is a semisimple, finite ribbon category. 
        \item There is a tensor equivalence between the category of $\W_{-1}(\so_{N},\minO)$-modules and the underlying category of $L_{-1+\frac{N-4}{2}}(\osp_{1|2})\otimes \mathcal{F}^{N-4}$-modules, including Ramond-twisted modules.
    \end{itemize}
\end{corollary}
A complete description of the category of weight $L_{-1+\frac{N-4}{2}}(\osp_{1|2})$-modules, including Ramond twisted modules, as a ribbon category can be inferred from the recent results on the category of weight $L_{-1+\frac{N-4}{2}}(\mathfrak{sl}_{2})$-modules \cite{ACK, C1, C2, CMY, NOCW}. 

\subsection*{Future directions}
The isomorphism in Theorem \ref{coset construction} is suggested by the building blocks \cite{CFLN24, CKL24} of the $\W$-algebras associated to Lie algebras of classical types. These building blocks are quotients of universal $\W$-algebras, denoted by $\W_{1+\infty}(c,\lambda), \W_{\infty}^{\mathrm{ev}}(c,\lambda)$ and $\W_{\infty}^{\spp_2}(c,\lambda)$, defined over the polynomial ring $\C[c,\lambda]$. In our set-up, the $\Z_2$-orbifolds
$$\mathrm{Com}(V^{k+2}(\so_{N-4}), \W^k(\so_N, \minO))^{\Z_2},\quad V^{\ell}(\osp_{1|2})^{\Z_2}$$
both appear as quotients of $\W_{\infty}^{\spp_2}(c,\lambda)$ at generic levels by imposing relations on $c$ and $\lambda$, called truncation curves \cite{CKL24}. Here, $\mathrm{Com}(V^{k+2}(\so_{N-4}),\W^k(\so_N, \minO))$ denotes the affine coset subalgebra inside $\W^k(\so_N, \minO)$. 
These curves intersect at the point corresponding to the levels $(k,\ell)=(-1,-1+\frac{N-4}{2})$ appearing in Theorem \ref{coset construction}, which suggests that their simple quotients are isomorphic. Moreover, when $N\geq4$ is even, they are common conformal extensions of strongly rational vertex algebras
$$\text{Com}(L_{k+2}(\so_{N-4}), \W_{k}(\so_N, \minO))^{\Z_2} \simeq L_{\ell}(\osp_{1|2})^{\Z_2}\supset L_\ell(\sll_2)\otimes \W_{-2+\frac{\ell+2}{2\ell+3}}(\sll_2,\oo{2}),$$
leading to the strong rationality of $ \W_{k}(\so_N, \minO)$.

Based on this perspective, there are two more families of $\W$-algebras of type $BD$ at integer non-admissible levels which we expect to be lisse or strongly rational, generalizing Theorem \ref{coset construction}.
Recall that the nilpotent orbits of $\so_N$ ($N\geq4$) are parametrized by partitions $\lambda$ of $N$ such that even parts have even multiplicities; in particular, $\OO_\min$ corresponds to the partition $[2^2,1^{N-4}]$. In what follows, we denote by $\OO_\lambda$ the nilpotent orbit corresponding to the partition $\lambda$.

\begin{conjecture}\label{conj} There exist isomorphisms of vertex algebras as follows:
        \begin{align*} 
        &\W_{1-2m}(\mathfrak{so}_{N}, \oo{2^{2m}, 1^{r}}) \simeq  (L_{-1 + \frac{r}{2} }(\mathfrak{osp}_{1|2m}) \otimes \mathcal{F}^r)^{\mathbb{Z}_2},
        &&(m\geq 1,r\geq 3),
        \\
        &\W_{1-2(m+s)}(\mathfrak{so}_{N}, \oo{3^s,2^{2m}, 1^{r}}) \simeq  (\W_{-1 + \frac{r-s}{2}}(\osp_{1|2(m+s)}, \oo{1|2^s,1^{2m}}) \otimes \mathcal{F}^r)^{\Z_2},
        &&(s\geq1,r\geq 3).
        \end{align*} 
\end{conjecture}

From the first isomorphism of Conjecture \ref{conj} and the strong rationality of $L_{k}(\mathfrak{osp}_{1|2m})$ at positive integer level $k\in\Z_{\geq0}$, we would get that $\W_{1-2m}(\mathfrak{so}_{N}, \mathbb{O}_{[2^{2m}, 1^{r}]})$ is strongly rational when $r$ is even, \ie\ in type $D$. 
We expect similarly that the second family of orbifolds will be strongly rational when $r$ is even, so this isomorphism will provide additional examples of strongly rational $\W$-algebras in type $D$ as well as $B$.
\vspace{1ex}

Recently, Shan--Yan--Zhao proposed a conjecture \cite{SYZ, SYZ1} which describes the associated variety of $L_k(\g)$, when $\g$ is simply-laced.
By assuming that the simple $\W$-algebra $\W_k(\g,\OO)$ can be recovered by the quantum Hamiltonian reduction $\HH_{\OO}^0(L_k(\g))$, it implies a conjectural list of lisse $\W$-algebras. 
In type $D$ at integer levels, the corresponding list is as follows
\begin{enumerate}
    \item $\W_{-2m}(\so_{2n},\oo{2^{2m},1^r})$\qquad $(m\geq1,\ r\geq 4)$,
    \item $\W_{1-2m}(\so_{2n},\oo{2^{2m},1^r})$\qquad $(m\geq1,\ r\geq 4)$,
    \item $\W_{-2(m+s)}(\so_{2n},\oo{3^{s},2^{2m},1^4})$\qquad $(s \geq4,\ m\geq 0)$.
\end{enumerate}

The $\W$-algebras of type (1) and several examples of type (3) (\eg\ $m=0$) are known to collapse to strongly rational affine vertex algebras (see for instance \cite{A+7}). We expect that the family (3) is a collapsing case in general. 
In addition, the family (2) appears in the first isomorphism of the Conjecture \ref{conj}; its strong rationality is predicted by the building block approach and the case $m=1$ is established in Corollary \ref{intro cor: strong rationality}. 
The family (3) is likewise related to the Conjecture \ref{conj} upon shifting the level by $-1$ in the second isomorphism.  
Based on these examples, it is natural to expect that these three families of $\W$-algebras are strongly rational. 
Furthermore, we expect that shifting the level by any positive integer preserves the lisse property of these $\W$-algebras.
We hope to come back to this point in future work.

\subsection*{Organization of the paper} 
The paper is organized as follows. In \S \ref{sec2} we use the uniqueness Theorem of minimal $\W$-algebras of \cite{ACKL} to prove Theorem \ref{coset construction}. In \S \ref{sec: ordinary modules} we classify all modules of $\W_{-1}(\so_{N},\minO)$ and Appendix \ref{appendix} is a collection of categorical statements that imply Corollary \ref{cor:res}.

\subsection*{Acknowledgements} 
We thank Tomoyuki Arakawa for drawing our attention to his work \cite{AM15} with A. Moreau, which triggered this project.
We are grateful to the referees for their valuable suggestions and for pointing us toward \cite{SYZ}.
T.C. is supported by DFG project Projektnummer 551865932.
J.F. was supported by a University of Melbourne Establishment Grant and an Andrew Sisson Support Package 2025.
V.K. is supported by Alexander von Humboldt Foundation. 
A.L. is supported by NSF Grant DMS-2401382 and Simons Foundation Grant MPS-TSM-00007694.
S.N. thanks the University of Denver for its hospitality during his stay.

\section{Orbifold realization of \texorpdfstring{$\W_{-1}(\so_N,\minO)$}{W-1(so(N),Omin)}}\label{sec2}

\subsection{Preliminary}
Given a vertex superalgebra $V$, let $V=V_{\overline{0}}\oplus V_{\overline{1}}$ be the parity decomposition, $\vac\in V_{\overline{0}}$ the vacuum vector, and
$$Y(\cdot,z)\colon V \rightarrow \mathrm{End}(V)\llbracket z^{\pm1}\rrbracket,\qquad a\mapsto a(z):=\sum_{n \in \Z} a_{(n)}z^{-n-1}$$
the vertex operator.
The translation operator and the normally ordered product are defined by
\begin{align*}
    \pd\colon V \rightarrow V,\ A \mapsto A_{(-2)}\vac,\qquad 
    V \times V\rightarrow V,\ (A,B) \mapsto A_{(-1)}B
\end{align*}
respectively, and the $\lm$-bracket by
\begin{align*}
   [\cdot{}_\lambda \cdot]\colon V\times V \rightarrow V[\lambda],\quad (A,B)\mapsto [A{}_\lambda B]:=\sum_{n=0}^\infty A_{(n)}B \tfrac{\lm^n}{n!}.
\end{align*} 

Let $\g$ be a simple Lie algebra, or the Lie superalgebra $\osp_{1|2n}$, and set $u_i$ ($i=1,\dots, \dim \g$) a basis of $\g$.
Then $V^k(\g)$ denotes the universal affine vertex (super)algebra associated with $\g$ at level $k \in \C$.
It is strongly and freely generated by the fields 
\begin{equation}\label{eq:field_affine_VOA}
u_i(z)=\sum_{n\in\Z}u_{i,(n)}z^{-n-1},\qquad u_{i,(n)}:=u_i\otimes t^n,
\end{equation}
\ie\ the PBW basis given by the monomials
\begin{equation}\label{PBW}
    u_{i_1,(-n_1)}u_{i_2,(-n_2)}\dots u_{i_\ell,(-n_\ell)}\vac
\end{equation}
with $i_1\geq i_2\geq\dots\geq i_\ell$, spans $V^k(\g)$.
Note that in \eqref{PBW}, if $i_j= i_{j+1}$, $n_j\geq n_{j+1}$ when $u_{i_j}$ is even whereas $n_j> n_{j+1}$ if $u_{i_j}$ is odd.
Their $\lm$-brackets are given by
\begin{align*}
    [u_i{}_\lm u_j]=k(u_i,u_j)\lm + [u_i,u_j]
\end{align*}
where $(\cdot,\cdot)$ is the normalized invariant bilinear form on $\g$. 
In this article, we always assume the level to be \emph{non-critical}, \ie\ $k\neq-h^\vee$ where $h^\vee$ is the dual Coxeter number of $\g$. Recall that we have $h^\vee=N-2$ for $\g=\so_N$ and $h^\vee=n-\frac{1}{2}$ for $\g=\osp_{1|2n}$. Then $V^k(\g)$ admits a conformal vector $L$ defined by the Sugawara construction, which satisfies the $\lambda$-bracket 
\begin{align*}
    \lbr{L}{L}=\frac{c_k}{2}\lm^3+2L \lm+ \pd L
\end{align*}
where the scalar $c_k=\frac{k\sdim\g}{k+h^\vee}$ is the central charge of $V^k(\g)$. The field $L$ generates a Virasoro vertex subalgebra at central charge $c_k$, denoted by $\Vir^{c_k}$.
Denote by $L_k(\g)$ the unique simple quotient of $V^k(\g)$ and by $\Vir_{c_k}$ the unique simple quotient of $\Vir^{c_k}$.

\subsection{Minimal \texorpdfstring{$\W$}{W}-algebras}\label{Minimal Walg}
Let $\g$ be a simple Lie algebra and take $f\in\g$ a (non-zero) nilpotent element.
By the Jacobson-Morozov theorem, one may find an $\sll_2$-triple $\{e,h=2x,f\}$ in $\g$, such that the adjoint $x$-action defines a $\frac{1}{2}\Z$-grading $\Gamma\colon \g=\bigoplus_{i}\g_i$.
The $\W$-algebra $\W^k(\g,f,\Gamma)$ associated with the pair $(f,\Gamma)$ is the quantum Hamiltonian reduction of the affine vertex algebra $V^k(\g)$.
Although $\W^k(\g,f,\Gamma)$ does not depend on the choice of $\Gamma$ as a vertex algebra, we consider only $\Gamma$ as the Dynkin grading, ie, the adjoint grading of $x$.
Then, $\W^k(\g,f,\Gamma)$ depends only on the adjoint class, \ie\ the nilpotent orbit $\OO$ containing $f$, and thus we denote it by $\W^k(\g,\OO)$ and its unique simple quotient by $\W_k(\g,\OO)$.

The nilpotent orbits form a poset where the partial order is defined by the closure relations for the Zariski topology on the nilpotent cone $\mathcal{N}$ of $\g$. The minimal nilpotent orbit $\OO_{\mathrm{min}}$ is the unique smallest, non-trivial orbit in $\mathcal{N}$. 
In this case, the grading $\Gamma$ is of the form
\begin{align*}
    \g=\g_{-1} \oplus \g_{-1/2} \oplus \g_{0} \oplus \g_{1/2} \oplus \g_{1}
\end{align*}
and $\W^k(\g,\minO)$ is strongly and freely generated by a basis of the centralizer 
$$\g^f=\g^f_0\oplus \g^f_{-1/2}\oplus \g^f_{-1}.$$
Note that $\g^\sharp=\g^f_0$ is a Lie subalgebra acting naturally on $\g^f_{-1/2}$ and $\g^f_{-1}$. In particular, $\g^f_{-1}=\C f$ is the trivial $\g^\sharp$-module.
By \cite[Theorem 5.1]{KW-QR},
\begin{itemize}
    \item[(M1)] $\g^\sharp$ generates an affine vertex subalgebra $V^{k^\sharp}(\g^\sharp)$ at level $k^\sharp$ that depends linearly on $k$,
    \item[(M2)] $\g^f_{-1}$ generates a Virasoro vertex subalgebra $\Vir^{c_{\W,k}}=\langle L_\W\rangle$ of central charge 
    $$c_{\W,k}=\frac{k \dim\g}{k+h^\vee}-6k+h^\vee-4,$$
    \item[(M3)] The generators $G^{\{u\}}$ ($u\in \g^f_{-1/2}$) have conformal weight $3/2$ and transforms as $\g^f_{-1/2}$ under the $\g^\sharp$-action, \ie\
    \begin{align*}
    \lbr{L_\W}{G^{\{u\}}}=\frac{3}{2}G^{\{u\}} \lm+\pd G^{\{u\}},\qquad \lbr{a}{G^{\{u\}}}=G^{\{[a,u]\}}\quad (a\in \g^\sharp).    
    \end{align*}
\end{itemize}
The remaining $\lambda$-brackets to fully determine $\W^k(\g,\minO)$ are given by
\begin{align}\label{OPE for weight3/2}
    \lbr{G^{\{u\}}}{G^{\{v\}}}=\frac{1}{6}a_0(k) (e,[u,v])\lm^2+ a_1(u,v;k)\lm+a_2(u,v;k)
\end{align}
with a certain scalar $a_0(k)$ depending only on $k$, and differential polynomials $a_i(u,v;k)$ of conformal weight $i=1,2$, see \cite[Theorem 5.1]{KW-QR} for the precise formula.

\begin{remark}
Note that the description in terms of generators of $\W^k(\g,\OO)$ does not depend on the choice of the representative $f\in\OO$. Indeed, all the representatives of $\OO$ are related by the action of the adjoint group and thus the system of generators is uniquely determined up to this action. The action of the adjoint group also preserves the conformal vector and the conformal grading.
\end{remark}

\begin{theorem}[{\cite[Theorem 3.2]{ACKL}}]\label{uniqueness of minimal Walg}
Let $V$ be a simple vertex algebra strongly generated by elements satisfying the properties (M1)-(M3). 
If in addition, $a_0(k)\neq0$ and the bilinear form 
\begin{equation*}
    V_{3/2}\times V_{3/2}\to \C,\qquad (A,B)\mapsto A_{(2)}B
\end{equation*}
is non-degenerate, then $V \simeq \W_k(\g,\minO)$ holds.
\end{theorem}

\subsection{Proof of Theorem \ref{coset construction}}
We fix a basis of $\osp_{1|2}$ whose corresponding fields \eqref{eq:field_affine_VOA} strongly and freely generate $V^\ell(\osp_{1|2})$. 
The even fields $e, f, h$ generate a vertex subalgebra $V^\ell(\sll_2)$ and the non-trivial $\lm$-brackets involving the odd fields $x, y$ of conformal weight one satisfy
\begin{equation*}
    \begin{aligned}
&\lbr{h}{x}= x, \qquad &&\lbr{h}{y} =  -y, \qquad &&\lbr{f}{x}=y, \qquad &&\lbr{e}{y}=x,\\
&\lbr{x}{x}=-2e, \qquad &&\lbr{y}{y}=2f,\qquad &&\lbr{x}{y}=2\ell \lm +h.
    \end{aligned}
\end{equation*}
Meanwhile, the $r$-free fermions vertex superalgebra $\mathcal{F}^r$ is strongly and freely generated by $\psi_1, \dots, \psi_r$ satisfying the $\lm$-brackets
$$\lbr{\psi_i}{\psi_j}=\delta_{i, j}.$$
Moreover, $\mathcal{F}^r$ has a conformal vector,
$$L_{\mathcal{F}^r}= \frac{1}{2}\sum_{i=1}^r (\partial\psi_i)\psi_i,$$
of central charge $c_r=\frac{r}{2}$, that gives $\psi_i$ ($1\leq i\leq r$) conformal weight $1/2$. 

Define the even vertex subalgebra
$$\scV^\ell(r):=(V^\ell(\osp_{1|2}) \otimes \mathcal{F}^r)_{\overline{0}}\subset V^\ell(\osp_{1|2}) \otimes \mathcal{F}^r$$
strongly generated by the subset 
$$\widetilde{\mathcal{S}}=\{e, f, h,X_{i}^{a, b},Y_{i}^{a, b},W_{i,j}^{a, b},Z_{z,z'}^{a, b}\mid 1\leq i,j\leq r,\, z, z' \in \{x, y\},\, a,b \geq0\}$$ 
where
$$X_{i}^{a, b}= (\partial^a x) (\partial^b \psi_i), \quad Y_{i}^{a, b} = (\partial^a y)(\partial^b \psi_i),\quad W_{i, j}^{a, b} =(\partial^a\psi_i) (\partial^b \psi_j), \quad Z_{z, z'}^{a, b} = (\partial^a z)(\partial^b z').$$

\begin{lemma}\label{lem:strong_gen}
For $r\geq3$, the vertex algebra $\scV^\ell(r)$ is strongly generated by the subset
$$\mathcal{S}=\{e, f, h, X_{m}^{0, 0}, Y_{m}^{0, 0}, W_{i, j}^{0, 0}, Z_{x,y}^{0, 0}\mid 1\leq i<j\leq r, 1\leq m\leq r\}\subset \widetilde{\mathcal{S}}.$$
\end{lemma}
\begin{proof}
Using the Leibniz rule $\partial (AB) = (\partial A)B + A (\partial B)$, the generating set of $\widetilde{\mathcal{S}}$ is reduced to 
$$e,\quad f,\quad h,\quad X_{i}^{a, 0},\quad Y_{i}^{a, 0},\quad W_{i, j}^{a, 0},\quad Z_{z, z'}^{a, 0}$$
for 
$$1\leq i,j\leq r,\quad a\geq0,\quad  z, z'\in\{x, y\}.$$
On the other hand, we have
\begin{equation*}
X_{i}^{a, 0} W_{i, m}^{0, 0} = (1- \delta_{i, m}) X_{m}^{a+1, 0}, 
\qquad Y_{i}^{a, 0} W_{i, m}^{0, 0} = (1- \delta_{i, m}) Y_{m}^{a+1, 0}.
\end{equation*}
for all $1\leq i,j,m\leq r$ and $a\geq0$.
Hence, by induction, one can exclude $X_{i}^{a, 0}$ and $Y_{i}^{a, 0}$ for $a>0$ from the strong generating set.

Now, we show that one can also remove $W_{i, j}^{a, 0}$ for $a>0$ from the strong generating set. 
It follows from direct computation that 
\begin{equation}\label{eq: induction on W}
W_{i, j}^{a, 0} W_{j, m}^{0, 0} = 
\begin{cases}
    W_{i, m}^{a+1, 0}\quad &(i,j,m \text{ all different}),\\[0.75em]
    \frac{(-1)^a}{a+1}W_{j,j}^{a+1,0}+W_{m,m}^{a+1, 0}\quad &(i=m,j\neq m),\\[0.75em]
    0\quad&(j=m).
\end{cases}
\end{equation}
holds for all $1\leq i,j,m\leq r$ and $a\geq0$.
Hence, one can exclude $W_{i, j}^{a, 0}$ with $i\neq j$ for $a>0$. 
The diagonal generators $W_{i, i}^{a, 0}$, with $a>0$, can also be removed since, using \eqref{eq: induction on W}, we have 
\begin{align*}
    W_{i,i}^{a+1,0}&=\begin{cases}
        \left(1- \frac{1}{(a+1)^2}\right)^{-1}\left(W_{i, j}^{a, 0} W_{j, i}^{0, 0}- \frac{(-1)^a}{a+1}W_{j, i}^{a, 0} W_{i, j}^{0, 0}\right)
    \qquad &(a>0),\\[1em]
    \frac{1}{2}(W_{i, j}^{0, 0} W_{j, i}^{0, 0}+W_{i, m}^{0, 0} W_{m,i}^{0, 0}-W_{m, j}^{0, 0} W_{j, m}^{0, 0})& (a=0).
    \end{cases}
\end{align*}
for $1\leq i,j,m\leq r$ all different.
Therefore, we only need $W_{i, j}^{0, 0}$ instead of $W_{i, j}^{a, 0}$ ($a\geq0$) as strong generators. Moreover, 
$$W_{i, j}^{0, 0} = \psi_i\psi_j=-\psi_j\psi_i=- W_{j, i}^{0, 0},$$ 
thus, one may further restrict to $W_{i, j}^{0, 0}$ with $i<j$ as in the assertion.

Finally, we consider the generations $Z_{z,z'}^{a,0}$. Since we have
\begin{equation*}
    \begin{split}
Z_{x, y}^{a+1, 0} = - X_{i}^{a, 0}Y_{i}^{0, 0}+\dots,\qquad
Z_{x, x}^{a+1, 0} = - X_{i}^{a, 0}X_{i}^{0, 0}+\dots,\qquad
Z_{y, y}^{a+1, 0} = - Y_{i}^{a, 0}Y_{i}^{0, 0}+\dots
\end{split}
\end{equation*}
where `$\dots$' are correction terms that only involve $W_{i, i}^{a, 0}$ and derivatives of $e,f,h$.
Thus, one may assume $a=0$ for the remaining $ Z_{z, z'}^{a, 0}$ as well. Because $Z_{z, z}^{0, 0}=0$ and $Z_{y,x}^{0, 0} = -Z_{x,y}^{0, 0}+\partial h$, one can restrict to $Z_{x,y}^{0, 0}$.
This completes the proof.
\end{proof}

\begin{proof}[Proof of Theorem \ref{coset construction} ]
The minimal nilpotent orbit $\minO\subset\so_{N}$ corresponds to the partition $[2^2,1^r]$ (with $r=N-4$) of $N$.
Applying Theorem \ref{uniqueness of minimal Walg}, we will show that the simple quotient 
$$\scV_\ell(r)\simeq (L_\ell(\osp_{1|2}) \otimes \mathcal{F}^r)_{\overline{0}},\qquad \ell=-1+\frac{r}{2},$$
is isomorphic to $\W_k(\so_{N},\minO)$ at level $k=-1$.

For $\W^k(\so_{N},\minO)$, $k=-1$, the data for (M1)-(M3) reads as 
\begin{align*}
  V^{k^\sharp}(\g^\sharp)\simeq V^{k_1^\sharp}(\sll_2) \otimes V^{k_2^\sharp}(\so_r),
  \quad \Vir^{c_{\W,k}} \subset \W^k(\so_{N},\minO)
\end{align*}
with 
$$  k_1^\sharp =-1+\frac{r}{2},\quad k_2^\sharp =1 ,\qquad c_{\W,-1} = -\frac{N(N-1)}{2(N-3)} +N=\frac{(r+4)(r-1)}{2(r+1)},$$
where $\so_1=0$, $\so_3\simeq\sll_2$, $\so_4\simeq \sll_2\oplus\sll_2$.
The elements of conformal weight $\frac{3}{2}$
form the natural basis of the standard representation of $\sll_2 \oplus \so_r$ and by \cite[Theorem 5.1]{KW-QR},
\begin{equation*}
    a_0(-1)=-3(N-3)\neq 0.
\end{equation*}

By Lemma \ref{lem:strong_gen}, $\scV^\ell(r)$ is strongly generated by the fields $e,f,h,W_{i, j}^{0, 0}, X_i^{0, 0}, Y_i^{0, 0}, Z_{x,y}^{0, 0}$.

Since, $\ell= k_1^\sharp$, of course, $e,f,h$ generates the vertex subalgebra $V^{k_1^\sharp}(\sll_2)$.
The $\frac{r(r-1)}{2}$ ($=\dim \so_r$) strong generators $W_{i, j}^{0, 0}$ ($1 \leq i < j \leq r$), that commute with $e,f,h$, all have conformal weight one.
Realising $\so_r$ as the set of matrices $\{X\in\gl_r\mid X^T=-X\}$ with the usual Lie bracket given by the commutator of matrices, one can identify $W_{i, j}^{0, 0}$ with $E_{i,j}-E_{j,i}$. 
Then, comparing the $\lambda$-brackets, we deduce that the $W_{i, j}^{0, 0}$'s generate a quotient of $V^1(\so_r)$.
Since they generate the simple vertex algebra $\mathcal{F}^r_{\overline{0}}$, hence they generate the simple quotient $L_1(\so_r)$.

Moreover, the calculation of the $\lambda$-brackets show that the $2r$ weight $\frac{3}{2}$ fields $X_m^{0, 0}$ and $Y_m^{0, 0}$ carry the structure of the standard representation of $\so_r \oplus \sll_2$. 
Finally, the conformal vector, constructed from the single conformal weight two generator $Z^{0, 0}_{x, y}$ and derivatives and normally ordered products of conformal weight one strong generators in $\scV_\ell(r)$, is given by
$$L_{\scV_\ell(r)}=L_{\osp_{1|2},\ell}+L_{\mathcal{F}^r}$$
and has central charge
$$c=\frac{2\ell}{3+2\ell}+\frac{r}{2}=c_{\mathcal{W},-1}.$$

It follows from Theorem \ref{uniqueness of minimal Walg} that  
$$\W_{-1}(\so_{N},\OO_{[2^2,1^r]}) \simeq (L_{-1 + \frac{r}{2}}(\osp_{1|2})\otimes \mathcal{F}^{r})_{\overline{0}}$$
as desired.
\end{proof}

\section{Simple Modules}\label{sec: ordinary modules}
In this section, we construct and classify the simple $\W_{-1}(\so_{N},\minO)$-modules.
This classification is first established in \S \ref{orbifold}, through the orbifold (Theorem \ref{coset construction}).
Then modules are realized in terms of the quantum Hamiltonian reduction in \S \ref{QHR-description}. 

\subsection{Generalities}
In this section, let $\g$ be a simple Lie algebra and $\widehat{\g}=\g[t^{\pm1}]\oplus \C K$ be the affine Lie algebra.
Fix
$$\g=\mathfrak{n}_+ \oplus \h \oplus \mathfrak{n}_-,\qquad \widehat{\g}=\widehat{\mathfrak{n}}_+ \oplus \widehat{\h} \oplus \widehat{\mathfrak{n}}_-$$
the usual triangular decompositions so that $\widehat{\h}=\h \oplus \C K$. 
Denote by $\alpha_1^\vee,\cdots,\alpha_n^\vee$ the simple coroots of $\g$ and by $\varpi_1,\cdots,\varpi_n$ its fundamental weights.
Then the simple coroots of $\widehat{\g}$ are 
$\alpha_0^\vee=-\theta+K, \alpha_1^\vee,\cdots,\alpha_n^\vee$
where $\theta$ is the highest root of $\g$ viewed as an element in $\h$ identified with its dual $\h^*$ through the normalized invariant bilinear form $(\cdot, \cdot)$. 

Let $\rho_\lm$ be the simple $\g$-module of highest weight $\lm \in \h^*$ and $L(\widehat{\lm})$ the simple $\widehat{\g}$-module of highest weight $\widehat{\lm} \in \widehat{\h}^*$.
The latter $L(\widehat{\lm})$ is naturally a (simple) positively graded $V^k(\g)$-module at level $k=\widehat{\lm}(K)$ whose the top space, 
that can be identified with the simple $\g$-module $\rho_\lm$ (with $\lm=\widehat{\lm}|_{\h^*}$),
has conformal weight 
\begin{align}
    h_\lm=\frac{({\lm},{\lm}+2{\rho})}{2(k+h^\vee)}
\end{align}
where ${\rho}=\sum_{i=1}^n\varpi_i$.
To highlight the structure of vertex algebra module, we denote $L(\widehat{\lm})$ by $L_k(\lm)$.

In the following, set $\OO=\minO$.
Given a $V^k(\g)$-module $M$, the zero-th BRST cohomology $\HH_{\OO}^0(M)$ associated with the minimal nilpotent orbit is naturally a $\W^{k}(\g,\OO)$-module. Indeed, this association gives a functor 
\begin{align*}
    \HH_\OO^0\colon V^k(\g)\mod\rightarrow \W^k(\g,\OO)\mod
\end{align*}
from the category of $V^k(\g)$-modules to the category of $\W^k(\g,\OO)$-modules.
By \cite{A05}, the simple ordinary $\W^k(\g,\OO)$-modules are classified by the highest weight representations $\mathbf{L}_k(\lm^\sharp,h_\lambda^\sharp)$ of heighest weight $\lm^\sharp=\lm|_{\h^\sharp}$ with respect to the Cartan subalgebra $\h^\sharp= \h\cap \g^\sharp\subset \g^\sharp$ and conformal dimension $h_\lambda^\sharp$ as follows.

\begin{theorem} \label{A05: minimal reduction}
For $\lm\in \h^*$, there is an isomorphism of $\W^k(\g,\OO)$-modules
\begin{align*}
    \HH^0_{\OO}(L_k(\lm)) \simeq \begin{cases}
        0 & \widehat{\lm}(\alpha_0^\vee)\in \Z_{\geq0}\\
        \mathbf{L}_k(\lm^\sharp, h_\lambda^\sharp) & \text{otherwise},
    \end{cases}
\end{align*}
where 
\begin{align}\label{conformal weight of minimal reduction}
    h_\lambda^\sharp= \frac{(\lm, \lm+2\rho)}{2(k+h^\vee)}-{(\lm,\theta)}.
\end{align}
\end{theorem}
\begin{remark}
Note that when the simple quotient $\W_{k}(\g,\OO)$ is strongly rational, as in our case for $\W_{-1}(\so_{N},\minO)$ when $N$ is even, all the finitely generated modules over $\W_{k}(\g,\OO)$ are ordinary: the conformal weights are bounded from below and they have finite-dimensional homogeneous subspaces with respect to the conformal grading.
\end{remark}

\subsection{Irreducible modules via orbifold}\label{orbifold}
Here, we construct all the simple $\W_{-1}(\so_{N},\minO)$-modules by using the realization established in Theorem \ref{coset construction}
\begin{align*}
   \W_{-1}(\so_{N},\minO)
    &\simeq (L_\ell(\osp_{1|2}) \otimes \mathcal{F}^r)_{\overline{0}}
    =L_\ell(\osp_{1|2})_{\overline{0}}\otimes \mathcal{F}^r_{\overline{0}}\oplus L_\ell(\osp_{1|2})_{\overline{1}}\otimes \mathcal{F}^r_{\overline{1}},
\end{align*}
with $\ell=-1+\frac{r}{2}$ and $r=N-4$.
By Theorem \ref{classification of simples}, simple modules can be obtained as 
\begin{align}\label{constructing simple modules}
    \ev{(M\otimes N)}=\ev{M}\otimes \ev{N} \oplus \od{M}\otimes \od{N},\quad \od{(M\otimes N)}=\ev{M}\otimes \od{N} \oplus \od{M}\otimes \ev{N}
\end{align}
where $M$ and $N$ are simple $L_\ell(\osp_{1|2})$-module and $\mathcal{F}^r$-module respectively, that are both local or both Ramond twisted.

Simple $L_\ell(\osp_{1|2})$-modules have been studied in \cite{CFK, CGL, CKLR2, RSW, wood} (see also \cite{KW88}). When working over vector superspaces, \eg\ the category of $L_\ell(\osp_{1|2})$-modules, we consider only morphisms of even parity to make the category abelian.
In particular, the complete sets of simple local and Ramond twisted $L_\ell(\osp_{1|2})$-modules are given by highest weight representations
\begin{align*}
    &\irr{loc}{L_\ell(\osp_{1|2})}=\obj{L_\ell(\lambda), \Pi L_\ell(\lambda) \mid  \lambda=0,1,\dots, \lfloor\tfrac{p-2}{2}\rfloor} ,\quad \\
    &\irr{R}{L_\ell(\osp_{1|2})}=\obj{L_\ell^R(\lambda), \Pi L_\ell^R(\lambda) \mid \lambda=0,1,\dots, \lfloor\tfrac{p-3}{2}\rfloor}
\end{align*}
where $\Pi L_\ell(\lambda)$ (resp.\ $\Pi L_\ell^R(\lambda)$) is the parity reversal of $L_\ell(\lambda)$ (resp.\ $L_\ell^R(\lambda)$) and
\begin{align*}
    p = \begin{cases}
        {2\ell+3} & (\ell \in \Z_{\geq0}),\\
        2(2\ell+3) & (\ell \in \mathbb Z_{\geq0} + \frac{1}{2}).
    \end{cases} 
\end{align*}

The parity decomposition of a (local or Ramond twisted) $L_\ell(\osp_{1|2})$-module is denoted by $M=\ev{M}\oplus\od{M}$ and, thus, $\ev{\Pi M}\simeq \od{M}$ and $\od{\Pi M}\simeq \ev{M}$ hold.
The top spaces of these direct summands are naturally modules over $\sll_2\subset \osp_{1|2}$.
By using the conformal weights of the top spaces of $L_\ell(\lambda)$ and $L_\ell^R(\lambda)$, respectively given by
$$h_\lambda = \frac{\lambda(\lambda+1)}{4(\ell+\frac{3}{2})},\qquad 
h^R_\lambda = \frac{(\lambda+\frac{1}{2})(\lambda+\frac{3}{2})}{4(\ell+\frac{3}{2})} -\frac{1}{8},$$
the $\sll_2$-module structure and the conformal weights of these top spaces are summarized in Table~\ref{tab:simple modules vira orbifold}.
\begin{table}[h!]
	\centering
	\renewcommand{\arraystretch}{1.3}
	\begin{tabular}{c|cccc}
		\hline
		& $\ev{L_{\ell}(\lm)}$ & $\od{L_{\ell}(\lm)}$ & $\ev{L_{\ell}^R(\lm)}$ & $\od{L_{\ell}^R(\lm)}$   \\ \hline
        top space & $\rho_{\lambda\omega}$ & \ttwolines{${\rho_{(\lambda-1)\omega}} $}{$(\lm \neq 0)$}{${\rho_{\omega}}$}{$(\lm = 0)$} & $\rho_{\lambda\omega}$ & \threelines{$\rho_{(\lambda+1)\omega}\oplus \rho_{(\lambda-1)\omega}$ }{$(\lm\neq 0,\ell)$}{$\rho_{\omega}$}{$(\lm=0)$}{$\rho_{(\ell-1)\omega}$}{$(\lm=\ell)$}\\ \hline
        $\Delta^{\mathrm{top}}$ & $h_\lm $ & \ttwolines{$h_\lm$}{$(\lm \neq 0)$}{$1$}{$(\lm = 0)$} & $h_\lambda^R$ & $h_\lambda^R+\tfrac{1}{2}$  \\ \hline
	\end{tabular}
 \captionsetup{font=small }
	\caption{The $\sll_2$-module structure and conformal weight of the top spaces. } 
    \label{tab:simple modules vira orbifold}
\end{table}

On the other hand, the $r$-free fermion vertex superalgebra $\mathcal{F}^r$ decomposes into
\begin{align*}
    \mathcal{F}^r\simeq L_1(0) \oplus L_1(\varpi_1)
\end{align*}
as a $L_1(\so_{r})$-module. 

The integer lattice $U=\mathbb Z \epsilon_1 \oplus \dots \oplus  \mathbb Z \epsilon_n$ equipped with the bilinear form such that $(\epsilon_i, \epsilon_j) = \delta_{i, j}$ provides a convenient realization for some data of $\so_r$:
\begin{itemize}
    \item (Type $D$) $r=2s$:
    \begin{align*}
        \alpha_i=\begin{cases}
            \epsilon_i - \epsilon_{i+1} & (1\leq i \leq s-2)\\
            \epsilon_{s-1} - \epsilon_{r'} & (i=s-1)\\
            \epsilon_{s-1} + \epsilon_{r'} & (i=s)
            \end{cases},\quad 
        \omega_i=\begin{cases}
            \epsilon_1+\cdots + \epsilon_{i}  & (1\leq i \leq s-2)\\
           \frac{1}{2}(\epsilon_1 + \dots + \epsilon_{s-1}-\epsilon_s) & (i=s-1)\\
            \frac{1}{2}(\epsilon_1 + \dots + \epsilon_{s-1}+\epsilon_s) & (i=s)
            \end{cases},
    \end{align*}
    and the highest root $\theta$ and the Weyl vector $\rho$ are respectively
    $$\theta = \epsilon_1+\epsilon_2,\qquad \rho = (s-1) \epsilon_1 + (s-2)\epsilon_2 + \dots + \epsilon_{s-1}.$$
    \item (Type $B$) $ r=2s+1$:
    \begin{align*}
        \alpha_i=\begin{cases}
            \epsilon_i - \epsilon_{i+1} & (1\leq i \leq s-1)\\
            \epsilon_{s} & (i=s)
            \end{cases},\quad 
        \omega_i=\begin{cases}
            \epsilon_1+\cdots +\epsilon_{i}  & (1\leq i \leq s-1)\\
            \frac{1}{2}(\epsilon_1 + \dots + \epsilon_{s-1}+\epsilon_s) & (i=s)
            \end{cases},
    \end{align*}
    and the highest root $\theta$ and the Weyl vector $\rho$ are respectively
    $$\theta = \epsilon_1+\epsilon_2,\qquad \rho = \tfrac{1}{2}\left( (2s-1) \epsilon_1 + (2s-3)\epsilon_2 + \dots + \epsilon_{s}\right).$$
\end{itemize}

Depending on the parity of $N=2n, 2n+1$ and thus $r=2n-4,2n-3$, the complete set of simple $L_1(\so_{r})$-modules is given by (see for instance \cite{Kac}):
\begin{itemize}
    \item (Type $D$) $N=2n$:
    $$\irr{}{L_1(\so_{r})}=\obj{L_1(0), L_1(\omega_1), L_1(\omega_{n-3}), L_1(\omega_{n-2})},$$
    whose conformal dimensions are $h_0=0$, $h_1=\frac{1}{2}$, $h_{n-3}=\frac{n-2}{8}$, $h_{n-2}=\frac{n-2}{8}$, respectively.
    \item (Type $B$) $N=2n+1$:
    $$\irr{}{L_1(\so_{r})}=\obj{L_1(0), L_1(\omega_1), L_1(\omega_{n-2})},$$
    whose conformal dimensions are $h_0=0$, $h_1=\frac{1}{2}$, $h_{n-2}= \frac{2n-3}{16}$, respectively.
\end{itemize}
Then $\mathcal{F}^r$ admits a unique simple local module $\mathbb{L}$ and a unique simple Ramond twisted module $\mathbb{L}^R$, whose parity decompositions are
\begin{align*}
    \mathbb{L}=L_1(0)\oplus L_1(\omega_1),\qquad \mathbb{L}^R=\begin{cases} L_1(\omega_{n-3})\oplus L_1(\omega_{n-2}) & (N=2n),\\    L_1(\omega_{n-2}) \oplus L_1(\omega_{n-2})& (N=2n+1). \end{cases}
\end{align*}

Now, we introduce the following simple $\W_{-1}(\so_{N},\minO)$-modules
\begin{align*}
      &\mathbb{L}^+(\lm)=\ev{(L_\ell(\lm)\otimes \mathbb{L}) },
      &&\mathbb{L}^-(\lm)=\od{(L_\ell(\lm)\otimes \mathbb{L}) },  \\
      &\mathbb{L}^{+}_R(\lm)=\ev{(L^R_\ell(\lm)\otimes \mathbb{L}^R) },  &&\mathbb{L}^{-}_R(\lm)=\od{(L^R_\ell(\lm)\otimes \mathbb{L}^R) }. 
\end{align*}
The classification of the simple $\W_{-1}(\so_{N},\minO)$-modules is now straightforward by using \eqref{constructing simple modules}.
\begin{theorem}\label{Classification of simples via orbifolds}
The complete set of simple $\W_{-1}(\so_{N},\minO)$-modules is given by 
\begin{itemize}
    \item (Type $D$) $N=2n$:
    $$\mathrm{Irr}=\obj{\mathbb{L}^+(\lm),\mathbb{L}^-(\lm) \mid  \lambda=0,1,\dots, n-3}\cup \obj{\mathbb{L}^+_R(\lm), \mathbb{L}^-_R(\lm)\mid  \lambda=0,1,\dots, n-3}.$$
    \item (Type $B$) $N=2n+1$:
    $$\mathrm{Irr}=\obj{\mathbb{L}^+(\lm),\mathbb{L}^-(\lm)\mid  \lambda=0,1,\dots, 2n-3}\cup \obj{\mathbb{L}^{-}_R(\lm)\mid  \lambda=0,1,\dots, 2n-4}.$$
\end{itemize}
\end{theorem}
Note that in the type $B$ case, we have $\mathbb{L}^{+}_R(\lm)\simeq \mathbb{L}^{-}_R(\lm)$ as $\W_{-1}(\so_{N},\minO)$-modules since $\mathbb{L}^R_{\overline{0}}\simeq \mathbb{L}^R_{\overline{1}}\simeq L_1(\varpi_{n-2})$.
We summarize, for each simple module, the structure of the top space and their conformal weights in Table \ref{simple modules vira orbifold} below.
\begin{table}[h!]
	\centering
	\renewcommand{\arraystretch}{1.3}
	\begin{tabular}{c|cccc}
		\hline
		& $\mathbb{L}^+(\lm)$ & $\mathbb{L}^-(\lm)$ & $\mathbb{L}^-_R(\lm)$ & $\mathbb{L}^+_R(\lm)~(N=2n)$   \\ \hline
        $\rho^{\mathrm{top}}$ & $\rho_{\lambda\omega} \otimes \rho_0$ & \ttwolines{$\rho_{(\lambda-1)\omega}\otimes \rho_0$}{$(\lm \neq0)$}{$\rho_0 \otimes \rho_{\omega_1}$}{$(\lm =0)$} &
        $\rho_{\lambda\omega} \otimes \rho_{\omega_{n-2}}$  & 
        $\rho_{\lambda\omega} \otimes \rho_{\omega_{n-3}}$\\ \hline
        $\Delta^{\mathrm{top}}$ & $h_\lm $ & \ttwolines{$h_\lm$}{$(\lm \neq0)$}{$1/2$}{$(\lm=0)$} & 
        $h_\lambda^R + h_{n-2}$ & $h_\lambda^R + h_{n-3}$ \\ \hline
	\end{tabular}
 \captionsetup{font=small }
	\caption{The $\sll_2$-module structure of top spaces of simple $\W_{-1}(\so_{N},\minO)$-modules and their conformal weights. }
    \label{simple modules vira orbifold}
\end{table}

\subsection{Irreducible modules via reduction}\label{QHR-description}
Using the explicit realization of the fundamental weights for $\so_N$ given in \S\ref{orbifold}, we obtain the following data on the conformal weights $h_\lambda^\sharp$ in \eqref{conformal weight of minimal reduction} at $k=-1$ that will be useful later:
\begin{itemize}
    \item (Type $D$) $N=2n$: 
    \begin{equation}\label{conformal weights D}
    \begin{split}
        &h^\sharp_{\lambda_1\omega_1 + \lambda_2\omega_2}
        = \frac{1}{2(2n-3)}\left( \lambda_1^2 + 2\lambda_2^2 + 2 \lambda_1\lambda_2 + \lambda_1 \right), \\
        &h^\sharp_{\lambda_1\omega_1 + \lambda_2\omega_2+ \omega_3}, 
        = \frac{1}{2(2n-3)}\left( \lambda_1^2 + 2\lambda_2^2 + 2 \lambda_1\lambda_2 + 3\lambda_1 + 4\lambda_2 \right) + \frac{1}{2},  \\
        &h^\sharp_{\lambda_1\omega_1 + \lambda_2\omega_2+ \omega_{n-1}} 
        = \frac{1}{2(2n-3)}\left( \lambda_1^2 + 2\lambda_2^2 + 2 \lambda_1\lambda_2 + 2\lambda_1 + 2\lambda_2 + \frac{2n^2-9n+12}{4} \right),   \\
        &h^\sharp_{\lambda_1\omega_1 + \lambda_2\omega_2+ \omega_{n}}
        = \frac{1}{2(2n-3)}\left( \lambda_1^2 + 2\lambda_2^2 + 2 \lambda_1\lambda_2 + 2\lambda_1 + 2\lambda_2 + \frac{2n^2-9n+12}{4} \right).
    \end{split}
\end{equation}
    \item (Type $B$) $N=2n+1$:
\begin{equation}\label{conformal weights B}
    \begin{split}
        &h^\sharp_{\lambda_1\omega_1 + \lambda_2\omega_2} 
        = \frac{1}{4(n-1)}\left( \lambda_1^2 + 2\lambda_2^2 + 2 \lambda_1\lambda_2 + \lambda_1  \right),  \\
        &h^\sharp_{\lambda_1\omega_1 + \lambda_2\omega_2+ \omega_3} 
        = \frac{1}{4(n-1)}\left( \lambda_1^2 + 2\lambda_2^2 + 2 \lambda_1\lambda_2 + 3\lambda_1 + 4\lambda_2  \right) + \frac{1}{2},  \\
        &h^\sharp_{\lambda_1\omega_1 + \lambda_2\omega_2+ \omega_{n}} 
        = \frac{1}{4(n-1)}\left( \lambda_1^2 + 2\lambda_2^2 + 2 \lambda_1\lambda_2 + 2\lambda_1 + 2\lambda_2 + \frac{2n^2-7n+8}{4} \right).
    \end{split}
\end{equation}
\end{itemize}

The simple $\W_{-1}(\so_{N},\minO)$-modules in Theorem \ref{Classification of simples via orbifolds} can be realized via the BRST reduction.
\begin{proposition}\label{simples via reduction}
There are isomorphisms of $\W_{-1}(\so_{N},\minO)$-modules
\begin{itemize}
    \item (Type $D$) $N=2n$:
    \begin{align*}
    &\mathbb{L}^+(\lm)\simeq \HH_{\OO}^0(L_{-1}(\lm \omega_1)),&
    &\mathbb{L}^-(\lm)\simeq \ccases{\HH_{\OO}^0(L_{-1}(\omega_3))}{(\lm=0)}{\HH_{\OO}^0(L_{-1}( (\lm-1)\omega_1+\omega_2))}{(\lm\neq0)},\\
    &\mathbb{L}^+_R(\lm)\simeq \HH_{\OO}^0(L_{-1}(\lm \omega_1+\omega_{n-1})),&
    &\mathbb{L}^-_R(\lm)\simeq  \HH_{\OO}^0(L_{-1}(\lm \omega_1+\omega_n)). 
    \end{align*}
    \item (Type $B$) $N=2n+1$:
    \begin{align*}
    &\mathbb{L}^+(\lm)\simeq \HH_{\OO}^0(L_{-1}(\lm \omega_1)),&
    &\mathbb{L}^-(\lm)\simeq \ccases{\HH_{\OO}^0(L_{-1}(\omega_3))}{(\lm=0)}{\HH_{\OO}^0(L_{-1}( (\lm-1)\omega_1+\omega_2))}{(\lm\neq0)},\\
    &&&\mathbb{L}^-_R(\lm)\simeq  \HH_{\OO}^0(L_{-1}(\lm \omega_1+\omega_n)).
    \end{align*}
\end{itemize}
\end{proposition}

\proof
For a dominant integral weight $\Lambda=\lm_1 \omega_1+\cdots + \lm_n \omega_n$ with $\lm_i\in \Z_{\geq0}$, $\widehat{\Lambda}(\alpha_0^\vee)=-(\Lambda,\theta)-1\notin \Z_{\geq0}$ holds.
Then Theorem \ref{A05: minimal reduction} implies that the simple modules $\mathbb{L}^\pm(\lm)$ (or $\mathbb{L}^\pm_R(\lm)$) is isomorphic to the highest weight module $\mathbf{L}_{-1}(\Lambda^\sharp, h_\Lambda^\sharp)\simeq \HH_{\OO}^0(L_{-1}(\Lambda))$,
for a certain dominant integral weight $\Lambda$,
if and only if the highest weights and conformal weights of the top spaces of both-hand sides are the same, \ie\
$$\text{(I)}\quad  \rho_{\Lambda^\sharp}\simeq \rho^{\mathrm{top}},\qquad \text{(II)}\quad h^\sharp_\Lambda=\Delta^{\mathrm{top}}.$$
. 
We do this case-by-case, using Table \ref{simple modules vira orbifold}, \eqref{conformal weights D} and \eqref{conformal weights B}, and 
\begin{align*}
    \Lambda^\sharp|_{\sll_2}=\lm_1 \omega,\qquad \Lambda^\sharp|_{\so_{N-4}}=\lambda_3 \omega_1 + \dots + \lambda_n \omega_{n-2}.
\end{align*}
\begin{itemize}[leftmargin=12pt]\setlength{\itemsep}{10pt}
    \item (Type $D$) $N=2n$:
\begin{enumerate}
    \item[(1)] $\mathbb{L}^+(\lm)$ with $0\leq \lm \leq n-3$:
    
    (I) implies $\Lambda = \lambda \omega_1 + \lambda_2 \omega_2$, and then (II) does $\lm_2=0$, \ie\ $\Lambda=\lm \omega_1$.
    
    \item[(2-1)] $\mathbb{L}^-(\lm)$ with $1\leq \lm \leq n-3$:

    (I) implies $\Lambda = (\lambda-1)\omega_1 + \lambda_2\omega_2$, and then (II) does $\lm_2=1$, \ie\ $\Lambda=(\lambda-1)\omega_1 + \omega_2$. 

    \item[(2-2)] $\mathbb{L}^-(\lm)$ with $\lm=0$:

   (I) implies $\Lambda =  \lambda_2\omega_2 + \omega_3$, and then (II) does $\lm_2=0$, \ie\ $\Lambda=\omega_3$.
   
    \item[(3)] $\mathbb{L}^+_R(\lm)$ with $1\leq \lm \leq n-3$:

    (I) implies $\Lambda = \lambda \omega_1 + \lambda_2 \omega_2 + \omega_{n-1}$, and then (II) does $\lm_2=0$, \ie\ $\Lambda=\lambda \omega_1 + \omega_{n-1}$.
    
    \item[(4)] $\mathbb{L}^-_R(\lm)$ with $1\leq \lm \leq n-3$:

    (I) implies $\Lambda = \lambda \omega_1 + \lambda_2 \omega_2 + \omega_{n}$, and then (II) does $\lm_2=0$, \ie\ $\Lambda= \lambda \omega_1 + \omega_{n}$. 
\end{enumerate}
\item (Type $B$) $N=2n+1$:
\begin{enumerate}
    \item[(1)] $\mathbb{L}^+(\lm)$ with $0\leq \lm \leq 2n-3$:
    
   (I) implies $\Lambda = \lambda \omega_1 + \lambda_2 \omega_2$, and then (II) does $\lm_2=0$, \ie\ $\Lambda=\lm \omega_1$.
   
    \item[(2-1)] $\mathbb{L}^-(\lm)$ with $1\leq \lm \leq 2n-3$:

    (I) implies $\Lambda = (\lambda-1)\omega_1 + \lambda_2\omega_2$, and then (II) does $\lm_2=1$, \ie\ $\Lambda=(\lambda-1)\omega_1 + \omega_2$.

    \item[(2-2)] $\mathbb{L}^-(\lm)$ with $\lm=0$:

    (I) implies $\Lambda =  \lambda_2\omega_2 + \omega_3$, and then (II) does $\lm_2=0$, \ie\ $\Lambda=\omega_3$.

    \item[(3)] $\mathbb{L}^-_R(\lm)$ with $1\leq \lm \leq 2n-4$:

    (I) implies $\Lambda = \lambda \omega_1 + \lambda_2 \omega_2 + \omega_{n}$, and then (II) does $\lm_2=0$, \ie\ $\Lambda= \lambda \omega_1 + \omega_{n}$. 
\end{enumerate}
\end{itemize}
This completes the proof.
\endproof

\appendix
\section{Categorical aspects of vertex superalgebras}\label{appendix}
Let $V$ be a simple vertex algebra equipped with a conformal vector and $\sC$ a vertex tensor category of $V$-modules. Hence, $\sC$ is a braided tensor category with a tensor product $\boxtimes$ and a unit $\one$. The associators and the braidings are respectively denoted by 
$$\mathcal{A}_{X,Y,Z}\colon X \boxtimes(Y \boxtimes Z)\simeq (X \boxtimes Y) \boxtimes Z,\quad \mathcal{B}_{X,Y}\colon X \boxtimes Y \simeq Y \boxtimes X.$$

Let $A=\ev{A}\oplus \od{A}$, with $\ev{A}=V$, be a simple vertex superalgebra in $\sC$. 
Then $J=\od{A}$ is a simple current, \ie\ a simple $V$-module such that $J \boxtimes J \simeq \one$. 
Indeed, $A$ is naturally identified with a commutative algebra object $\ub{A}=(\one, J)$ in the larger category  $\ub{\sSC}$ defined as
the braided tensor \emph{super}category 
whose objects are pairs $\bX=(X_0, X_1)$ in $\sC\times\sC$ -- the parity decomposition of $\bX$\ -- with even morphisms $f\colon \bX\rightarrow \bY$, 
$$f=(f_0, f_1),\qquad f_0: X_0 \rightarrow Y_0,\quad f_1: X_1 \rightarrow Y_1.$$
The unit in $\ub{\sSC}$ is $\ub{\one}=(\one,0)$, the tensor product is again $\boxtimes$ preserving the parity,
$$\bX\boxtimes\bY=(X_0\boxtimes Y_0\oplus X_1\boxtimes Y_1,X_0\boxtimes Y_1\oplus X_1\boxtimes Y_0),$$
and the braiding is modified to respect the parity,
$$\mathcal{B}_{\bX,\bY}\colon \bX \boxtimes \bY \simeq \bY \boxtimes \bX,\quad \mathcal{B}_{\bX,\bY}|_{X_p\boxtimes Y_q} =(-1)^{p q} \mathcal{B}_{X_p,Y_q}.$$

Following \cite{CKL20, CKM}, denote by $\ub{\sSC}_A$ the category of $\ub{A}$-module objects lying in $\ub{\sSC}$ equipped with parity-preserving morphisms. 
The objects in $\ub{\sSC}_A$ are pairs $(\bX, m_\bX)$ with $\bX$ in $\ub{\sSC}$ equipped with a structure map
$$m_\bX: \ub{A} \boxtimes \bX \rightarrow \bX$$
satisfying compatibility conditions with the multiplication $\mu\colon \ub{A} \boxtimes \ub{A} \rightarrow \ub{A}$.
The morphisms ${f: \bX \rightarrow \bY}$ in $\ub{\sSC}_A$ are those in $\ub{\sSC}$ that commute with the structure maps, \ie\ $m_\bY \circ (\text{id}_{\ub{A}} \boxtimes f) = f \circ m_\bX$. 

Since $J \boxtimes J \simeq \one$, {arbitrary $\ub{\sSC}_A$-modules are direct sums of two objects $\bX \oplus \bY$} satisfying the following properties:
\begin{equation*}
\vcenter{\xymatrix@R=1em{ \ub{J} \boxtimes \bX  \ar[rr]^-{\mathcal{M}_{\ub{J},\bX}} \ar[rd]_{m_\bX}&& \ub{J} \boxtimes \bX   \ar[ld]^{m_\bX} &\ub{J} \boxtimes \bY  \ar[rr]^-{(-1)^\bullet \mathcal{M}_{\ub{J},\bY}} \ar[rd]_{m_\bY}&& \ub{J} \boxtimes \bY   \ar[ld]^{m_\bY}\\ & \bX & && \bY &.}}
\end{equation*}
where $\ub{J}=(0,J)$, $\mathcal{M}_{\ub{J},\bullet}=\mathcal{B}_{\bullet,\ub{J}}\circ \mathcal{B}_{\ub{J},\bullet}$ is the monodromy, and $(-1)^\bullet$ denotes the parity.
The objects satisfying the first property are called the local (or Neveu-Schwarz) $\ub{A}$-modules corresponding to modules over the vertex superalgebra $A$, while the objects satisfying the second property are called twisted (or Ramond) $\ub{A}$-modules. Thus, $\ub{\sSC}_A$ decomposes into two full subcategories 
\begin{align*}
    \ub{\sSC}_A=\ub{\sSC}_A^{\mathrm{loc}} \oplus \ub{\sSC}_A^{\mathrm{tw}},
\end{align*}
consisting of local $\ub{A}$-modules and twisted $\ub{A}$-modules. 
By \cite{CKM}, $\ub{\sSC}_A$ is a tensor (super)category and $\ub{\sSC}_A^{\mathrm{loc}}$ is a braided tensor category.

Note that we have an induction functor
$$
\mathcal F: \sC \rightarrow \ub{\sSC}_A,\qquad
X \mapsto (\ub{A} \boxtimes (X, 0), \mu \boxtimes (\id_X, 0)), \qquad
f \mapsto \id_{\ub{A}} \boxtimes (f, 0),
$$
whose right adjoint is the forgetful functor $\mathcal G: \ub{\sSC}_A \rightarrow \sC$. 
In particular, Frobenius reciprocity holds
\begin{align*}
    \Frob_{\mathcal F(X), (\bY, m_\bY)}:\Hom_{\ub{\sSC}_A}(\mathcal F(X), (\bY, m_\bY))\overset{\simeq}{\longrightarrow}
\Hom_{\ub{\sSC}}((X, 0), \bY) \simeq 
\Hom_{\sC}(X, Y_0).
\end{align*}

\begin{theorem}\label{classification of simples}
Let $A=V \oplus J$ be a simple vertex superalgebra in a braided tensor category $\sC$ of $V$-modules as above. Then, the induction functor  $\mathcal F: \sC \rightarrow \ub{\sSC}_A$ gives an equivalence of tensor categories. Thus, there is a bijection on the sets of inequivalent simple objects $\mathrm{Irr}(\sC)\simeq \mathrm{Irr}(\ub{\sSC}_A)$.
\end{theorem}
\proof
It suffices to show that $\mathcal{F}$ is fully faithful and essentially surjective.
The former follows from Frobenius reciprocity $\Frob_{\mathcal F(X), \mathcal F(Y)}$:
\[
\Hom_{\ub{\sSC}_A}(\mathcal F(X), \mathcal F(Y))
\simeq\Hom_{\ub{\sSC}}((X, 0), \ub{A} \boxtimes (Y, 0))
\simeq\Hom_{\sC}(X, Y). 
\]
Let us show the latter. Take an $A$-module $(\bX, m_\bX)$ in $\ub{\sSC}_A$. We show the canonical map 
$$\Frob_{\mathcal F(X), (\bX, m_\bX)}^{-1}(\id_{X_0})\colon \mathcal{F}(X_0)\rightarrow \bX$$
is an isomorphism, that is $m_\bX \colon J\boxtimes X_0\rightarrow X_1$ is an isomorphism.
Indeed, the associativity of the structure map $m_\bX: \ub{A} \boxtimes \bX \rightarrow \bX$ restricted to $\ub{J} \boxtimes (\ub{J} \boxtimes \bY)\subset \ub{A} \boxtimes (\ub{A} \boxtimes \bY)$ with $\ub{J} = (0, J)$ gives the following commutative diagram:
\begin{equation}
\vcenter{\xymatrix{ \ub{J} \boxtimes (\ub{J} \boxtimes \bX)  \ar[rr]^-{\mathcal A}_\simeq \ar[d]_{\id \boxtimes m_\bX}&& (\ub{J} \boxtimes \ub{J}) \boxtimes \bX \ar[rr]^-{\mu \boxtimes \id}_\simeq && \ub{\mathbf 1} \boxtimes \bX \ar[d]^{m_\bX}_\simeq    \\ \ub{J} \boxtimes \bX \ar[rrrr]_-{m_\bX }&&&& \bX.}}
\end{equation}
Hence $m_{\bX}\circ(\id\boxtimes m_{\bX})\colon\ub{J} \boxtimes (\ub{J} \boxtimes \bX)\to\bX$ is an isomorphism and so is $m_{\bX}\colon J\boxtimes \bX\to \bX$ 
(the injectivity follows from $J\boxtimes J \simeq \one$). 
Since $\bX=(X_0,X_1)$, it follows 
\begin{align*}
        m_\bX\colon J\boxtimes X_0\xrightarrow{\simeq} X_1,\quad J\boxtimes X_1\xrightarrow{\simeq} X_0
\end{align*}
as desired. 
\endproof

Here, we record a nice criterion regarding the category $\sC(\g, k)$ of ordinary modules over simple affine vertex algebras $L_k(\g)$.
Recall that a tensor category $\sC$ is said of moderate growth if each indecomposable object $X$ satisfies $\dim \End(X^{\boxtimes n}) < n!$ for some $n\in\Z_{>0}$, and called an $r$-category if for every object $X$, the functor $\Hom_\sC(\bullet \boxtimes X, \one)$ is represented by some object $X^*$, giving rise to an equivalence $\sC \xrightarrow{\simeq}\sC^{mop},\ X\mapsto X^*$, to the tensor and arrows opposite category $\sC^{mop}$, see \cite{EP} and references therein.

\begin{theorem} \label{thm: general theorems from literature}
\phantom{x}
\begin{enumerate}[left=\itemindent]
    \item \textup{\cite[Theorem 2.12]{ALSW}}
Let $V$ be a simple, self-dual vertex algebra with a conformal vector and $\sC$ be a vertex tensor category of $V$-modules. If $\sC$ is closed under the contragredient dual, then $\sC$ is an $r$-category. 
\item \textup{\cite[Corollary 1.3]{EP}} Every semisimple braided $r$-category of moderate growth is rigid. 
\end{enumerate}
\end{theorem}
\begin{corollary}\label{cor:FTC}
For a simple Lie (super)algebra $\g$, let $\sC(\g, k)$ be the category of ordinary $L_k(\g)$-modules at non-critical level $k$. If $\sC(\g, k)$ is finite and semisimple, then it is a vertex ribbon category.
\end{corollary}
\begin{proof}
By assumption, $\sC(\g, k)$ is semisimple and thus forms a vertex tensor category by \cite[Corollary 3.4.5]{CY}. Thanks to the general property of vertex tensor categories, $\sC(\g, k)$ is a balanced braided tensor category with twist given by $e^{2\pi i L_0}$ where $L_0$ is the zero-mode of the conformal vector. Thus, it suffices to show that $\sC(\g, k)$ is rigid. It is clear that $L_k(\g)$ and $\sC(\g, k)$ satisfy the assumptions of Theorem \ref{thm: general theorems from literature} (1), $\sC(\g, k)$ is an $r$-category and, thus, it suffices to show that $\sC(\g, k)$ is of moderate growth by Theorem \ref{thm: general theorems from literature} (2). 

As $\sC(\g, k)$ is finite, let $M_1, \dots, M_r$ be the complete set of simple ordinary $L_k(\g)$-modules and $N_{i, j}^{\ \ \ell}$ be the fusion rules, \ie\
    \[
    M_i \boxtimes M_j \simeq \bigoplus_{\ell =1}^r N_{i, j}^{\ \ \ell} M_\ell.
    \]
    Set $m =r \cdot \text{max}\{ (N_{i, j}^{\ \ \ell})^2 | 1 \leq i, j, \ell \leq r \}$. By Schur's Lemma, it follows that $\dim\End(M_i \boxtimes M_j) = \sum_{\ell = 1}^r (N_{i, j}^{\ \ \ell})^2 \leq m$. By induction, we find that $\dim\End(M_{i_1} \otimes \dots \otimes M_{i_n})\leq m^{n-1}$. Since $m^{n-1}< n!$ holds for large enough $n$, one has 
    $\dim\End(M_i^{\otimes n}) < n!$ for some $n$, \ie\ $M_i$ is of moderate growth. This completes the proof.
\end{proof}

\end{document}